\documentclass[reqno]{amsart}

\usepackage{amsfonts}
\usepackage{amssymb}
\usepackage{amsthm}
\usepackage{setspace}
\usepackage{amsmath}
\usepackage{mathtools}
\usepackage{enumerate}
\usepackage{graphics}
\usepackage{fancybox}

\newtheorem{thm}{Theorem}[section]

\theoremstyle{definition}
\newtheorem{Def}[thm]{Definition}
\newtheorem{ex}[thm]{Example}
\newtheorem{rmk}[thm]{Remark}

\makeatletter \renewenvironment{proof}[1][\proofname] {\par\pushQED{\qed}\normalfont\topsep6\p@\@plus6\p@\relax\trivlist\item[\hskip\labelsep\bfseries#1\@addpunct{.}]\ignorespaces}{\popQED\endtrivlist\@endpefalse} \makeatother

\begin{document}
\title{Average chain transitivity and the almost average shadowing property}
\author{Mukta Garg \and Ruchi Das}
\newcommand{\acr}{\newline\indent}
\address{(M. Garg) Department of Mathematics, University of Delhi, Delhi-110007, India}
\email{mgarg@maths.du.ac.in, mukta.garg2003@gmail.com}

\address{(R. Das) Department of Mathematics, University of Delhi, Delhi-110007, India}
\email{rdasmsu@gmail.com}


\begin{abstract}

In this paper, we introduce and study notions of average chain transitivity, average chain mixing, total average chain transitivity and almost average shadowing property. We also discuss their interrelations.\end{abstract}
\pagestyle{plain}
\maketitle

\section{Introduction}
The basic goal of the theory of discrete dynamical systems is to describe the nature of all trajectories of the system. However, in certain instances, it is almost impossible to compute the exact initial value of $x$, which further leads to the approximate values of $f(x)$, $f^2(x)$ and so on. In this process, instead of the actual trajectory, we obtain an approximate trajectory of the system at $x$, which is named as pseudo trajectory (or pseudo-orbit). The idea of putting these pseudo trajectories close to the true trajectories of the system motivates the theory of shadowing property.

The shadowing property holds a significant portion of the theory of dynamical systems because of its close relation to the stability and to the chaoticity of the system. There are now various variants of this concept which exist in the literature, for instance, one can refer \cite{B, C, X, Gu, P} and some equivalences are obtained for expansive homeomorphisms having shadowing property on compact metric spaces \cite{K, L, Ks}.

One of the most significant properties of discrete dynamical systems is topological transitivity. In topological transitive systems, any point $x$ in the phase space can be reached to any other point $y$ of the space via true orbit of some point in any neighborhood of $x$. The notion of chain transitivity is a natural generalization of the notion of topological transitivity which connects any pair of points in the phase space by a pseudo-orbit with any desired error bound. There might be a situation that given an error bound $\delta$, we are not able to obtain a $\delta$-pseudo-orbit but it may be easier to find an $\eta$-average-pseudo-orbit with any given average error bound $\eta$.

This motivates us to introduce the notion of average chain transitivity, which is weaker than the notion of chain transitivity. We also introduce the notion of almost average shadowing property and study its relation with the average chain transitivity. The paper is organized as follows. In Section 2, we give notations and necessary definitions required for remaining sections. In Section 3, we introduce notions of average chain transitivity, average chain mixing, total average chain transitivity and study the relations among them. Section 4 is devoted to the introduction of the notion of almost average shadowing property (ALASP) and to the study of some of its properties. Here we also investigate the relation of the ALASP with chain transitivity, chain mixing and chain components.
\section{Preliminaries}

Throughout this paper $\mathbb{N}$ denotes the set of natural numbers and $\mathbb{Z}_+$ denotes the set of nonnegative integers. By a dynamical system, we mean a pair $(X,f)$, where $X$ is a metric space with metric $d$ and $f:X\rightarrow X$ is a continuous map. A sequence $\{x_i\}_{i\geqslant 0}$ in $X$ is called an orbit of $f$ if $x_{i+1}=f(x_i)$ for every $i\geqslant 0$ and for $\delta >0$, $\{x_i\}_{i\geqslant 0}$ is called a $\delta$-pseudo-orbit of $f$ if $d(f(x_i),x_{i+1})<\delta$ for every $i\geqslant 0$. The map $f$ is said to have the \textit{shadowing property} \cite{AH} if for every $\epsilon>0$, there is a $\delta>0$ such that every $\delta$-pseudo-orbit $\{x_i\}_{i\geqslant 0}$ is $\epsilon$-shadowed by some point $z\in X$, that is, $d(f^i(z),x_i)<\epsilon$ for every $i\geqslant 0$. For $\delta>0$, a sequence $\{x_i\}_{i\geqslant 0}$ in $X$ is called a $\delta$-average-pseudo-orbit of $f$ if there is an integer $N=N(\delta)>0$ such that for all $n\geqslant N$ and all $k\geqslant 0$,

\begin{center}
$\displaystyle\frac{1}{n}\sum\limits_{i=0}^{n-1}d(f(x_{i+k}),x_{i+k+1})<\delta$.
\end{center}

The map $f$ is said to have the \textit{average-shadowing property} (ASP) \cite{B} if for every $\epsilon>0$, there is a $\delta>0$ such that every $\delta$-average-pseudo-orbit $\{x_i\}_{i\geqslant 0}$ of $f$ is $\epsilon$-shadowed in average by some point $z\in X$, that is,

\begin{center}
$\limsup\limits_{n\rightarrow\infty}\displaystyle\frac{1}{n}\sum\limits_{i=0}^{n-1}d(f^i(z),x_i)<\epsilon$.
\end{center}

Recall that for $\delta>0$ and $x$, $y\in X$, a \textit{$\delta$-chain} of $f$ from $x$ to $y$ of \textit{length} $n\in\mathbb{N}$ is a finite sequence $x_0=x,x_1,\dots,x_n=y$ satisfying $d(f(x_i),x_{i+1})<\delta$ for $0\leqslant i\leqslant n-1$. Two points $x$, $y\in X$ are called \textit{chain equivalent} if for every $\delta>0$, there exist a $\delta$-chain of $f$ from $x$ to $y$ and a $\delta$-chain of $f$ from $y$ to $x$. A point $x\in X$ is called \textit{chain recurrent point} of $f$ if $x$ is chain equivalent to itself. We denote the set of all chain recurrent points of $f$ by $CR(f)$ and call it as \textit{chain recurrent set} of $f$. It is clear that the relation of being chain equivalent is an equivalence relation on $CR(f)$. An equivalence class under this equivalence relation is called a \textit{chain component} of $f$ \cite{AH}. The map $f$ is said to be \textit{chain transitive} if for any $\delta>0$ and any pair $x$, $y\in X$, there is a $\delta$-chain of $f$ from $x$ to $y$ and it is \textit{totally chain transitive} if each $f^k$, $k\in\mathbb{N}$, is chain transitive. The map $f$ is said to be \textit{chain mixing} if for any $\delta>0$ and any pair $x$, $y\in X$, there exists $N\in\mathbb{N}$ such that for any $n\geqslant N$, there is a $\delta$-chain of $f$ from $x$ to $y$ of length $n$ \cite{R}.

We also recall that the map $f$ is \textit{topologically transitive} \cite{AH} if for any pair of nonempty open sets $U$, $V\subseteq X$, there is an $n\in \mathbb{Z}_+$ such that $f^n(U)\cap V\ne\emptyset$ and it is \textit{topologically mixing} if for any pair of nonempty open sets $U$, $V\subseteq X$, there is an $N\in\mathbb{N}$ such that $f^n(U)\cap V\ne\emptyset$ for all $n\geqslant N$. It is well known that if $f$ has the shadowing property, then topological transitivity coincides with chain transitivity and topological mixing coincides with chain mixing. For any $A\subseteq\mathbb{Z}_+$, we define the \textit{upper density} of $A$ by

\begin{center}
$u_d(A)=\limsup\limits_{n\rightarrow\infty}\displaystyle\frac{1}{n}|A\cap\{0,1,\dots,n-1\}|$,
\end{center}
where $|\cdot|$ denotes the cardinality of the set.

A set $A\subseteq \mathbb{Z}_+$ is said to be \textit{syndetic} if it has bounded gaps, that is, there exists $N\in\mathbb{N}$ such that $[n,n+N]\cap A\ne\emptyset$ for every $n\in\mathbb{Z}_+$. The map $f$ is said to be \textit{strongly ergodic} if for any pair of nonempty open sets $U$, $V\subseteq X$, the set $N(U,V)=\{n\in\mathbb{Z}_+:f^n(U)\cap V\ne\emptyset\}$ is syndetic. The map $f$ is \textit{totally strongly ergodic} if each $f^k$, $k\in\mathbb{N}$, is strongly ergodic. A point $x\in X$ is said to be \textit{minimal} if for every neighborhood $U$ of $x$, the set $N(x,U)=\{n\in\mathbb{Z}_+:f^n(x)\in U\}$ is syndetic.

\section{Average chain transitivity and Average chain mixing}

Let $(X,f)$ be a dynamical system.
\begin{Def}
For $\delta>0$ and $x,y\in X$, a \textit{$\delta$-average-chain} of $f$ from $x$ to $y$ of \textit{length} $n\in\mathbb{N}$ is a finite sequence $x_0=x,x_1,\dots,x_n=y$ for which there exists $N\in\mathbb{N}$, $N\leqslant n$, such that for all $N\leqslant m\leqslant n$,
\begin{center}
$\displaystyle\frac{1}{m}\sum\limits_{i=0}^{m-1}d(f(x_i),x_{i+1})<\delta$.
\end{center}
\end{Def}

\begin{Def}
The map $f$ is said to be \textit{average chain transitive} if for any $\delta>0$ and any pair $x$, $y\in X$, there is a $\delta$-average-chain of $f$ from $x$ to $y$ of some length $n$. The map $f$ is said to be \textit{totally average chain transitive} if each $f^k$, $k\in\mathbb{N}$, is average chain transitive.
\end{Def}

\begin{Def}
The map $f$ is said to be \textit{average chain mixing} if for any $\delta>0$ and any pair $x$, $y\in X$, there exists $n_0\in\mathbb{N}$ such that for any given $n\geqslant n_0$, there is a $\delta$-average-chain of $f$ from $x$ to $y$ of length $n$.
\end{Def}

Clearly, the following implications hold: transitivity$\implies$ chain transitivity$\implies$ average chain transitivity. Also mixing$\implies$ chain mixing$\implies$ average chain mixing.

By the above implications we have that tent map on unit interval, doubling map on circle, irrational rotations on circle, identity map on a connected metric space are average chain mixing and adding machine on Cantor space (see \cite{R}) is average chain transitive.

Following are some examples of maps which are average chain mixing but not chain transitive.

\begin{ex}\label{ex}
Let $(X,d)$ be a metric space with more than one element and $a\in X$. Define $f:X\rightarrow X$ by $f(x)=a$ for every $x\in X$. Then clearly $f$ is not chain transitive since for any pair $x$, $y\in X$ with $y\ne a$, there is no $\delta$-chain of $f$ from $x$ to $y$ with $\delta<d(y,a)$. To see that $f$ is average chain mixing, take $\delta>0$ and $x$, $y\in X$. The case $y=a$ is trivial. Suppose $y\ne a$, say $d(a,y)=\epsilon>0$, choose $n_0\in\mathbb{N}$ such that $n_0>\epsilon/\delta$. For any $n\geqslant n_0$, define $x_i=f^i(x)$, $0\leqslant i\leqslant n-1$, $x_n=y$. Then the finite sequence $\{x_0,x_1,\dots,x_n\}$ satisfies $\frac{1}{m}\sum_{i=0}^{m-1}d(f(x_i),x_{i+1})<\delta$ for every $1\leqslant m\leqslant n$ which implies $f$ is average chain mixing.
\end{ex}

\begin{ex}
Let $X$ be a union of two disjoint circles with metric any $d$. Consider the identity map, $I_X$, on $X$. It is clear that $I_X$ is not chain transitive. However, one can prove that $I_X$ is average chain mixing (as done in Example \ref{ex}).
\end{ex}

\begin{ex}\label{ex1}
Consider $X=\{1,2,3,\dots,2k\}$ for some fixed $k\in\mathbb{N}$, with discrete metric $d$ and define $f:X\rightarrow X$ by $f(i)=(i+2)\mbox{mod}\hspace{0.5mm}2k$.  Clearly, $f$ is not chain transitive for if $i\in X$ is an odd number and $j\in X$ is an even number, then there does not exist any 1/2-chain of $f$ from $i$ to $j$. However, one can prove that $f$ is average chain mixing (as done in Example \ref{ex}).
\end{ex}

\begin{thm}
Let $(X,f)$ be a dynamical system. If $f^k$ is average chain transitive for some $k>1$, then so does $f$.
\end{thm}
\begin{proof}
Let $\delta>0$ and $x$, $y\in X$. Since $f^k$ is average chain transitive, there exists a $\delta$-average-chain, say $\{y_i\}_{0\leqslant i\leqslant n}$, of $f^k$ from $x$ to $y$. So there exists $N\in\mathbb{N}$, $N\leqslant n$, such that for all $N\leqslant m\leqslant n$,
\begin{center}
$\displaystyle\frac{1}{m}\sum\limits_{i=0}^{m-1}d(f^k(y_i),y_{i+1})<\delta$.
\end{center}
Define 
\begin{center}
$x_{pk+j}=\left\{\begin{array}{ll} f^j(y_p),\hspace{1mm} 0\leqslant p\leqslant n-1, 0\leqslant j\leqslant k-1,\\ y_n,\hspace{7.5mm} p=n, j=0,\end{array}\right.$
\end{center}
that is, $\{x_i\}_{0\leqslant i\leqslant nk}=\{y_0=x, f(y_0),\dots, f^{k-1}(y_0), y_1, f(y_1),\dots, f^{k-1}(y_1),\dots,y_{n-1},\\f(y_{n-1}),\dots, f^{k-1}(y_{n-1}), y_n=y\}$.
Then for all $n\leqslant r\leqslant nk$,
\begin{center}
$\displaystyle\frac{1}{r}\sum\limits_{i=0}^{r-1}d(f(x_i),x_{i+1})\leqslant\frac{1}{n}\sum\limits_{i=0}^{n-1}d(f^k(y_i),y_{i+1})<\delta$.
\end{center}
Thus $f$ is average chain transitive.
\end{proof}

\begin{thm}
Let $(X,f)$ be a dynamical system and $f$ be Lipschitz function. If $f$ is average chain mixing, then it is totally average chain transitive.
\end{thm}
\begin{proof}

Let $k>1$, $\delta>0$ and $x$, $y\in X$. Since $f$ is Lipschitz function, there exists $L>0$ such that $d(f(a),f(b))\leqslant L\hspace{0.5mm}d(a,b)$ for all $a$, $b\in X$. Without loss of generality, we can assume that $L\geqslant 1$. Choose $\epsilon=\delta/kL^{k-1}$. Since $f$ is average chain mixing, there exists $n_0\in\mathbb{N}$ such that for any $n\geqslant n_0$, there is an $\epsilon$-average-chain of $f$ from $x$ to $y$ of length $n$. Choose $m>0$ such that $mk\geqslant n_0$ so that $\{x_0=x,x_1,\dots,x_{mk}=y\}$ is an $\epsilon$-average-chain of $f$ from $x$ to $y$. Then there exists $N\in\mathbb{N}$, $N\leqslant mk$, such that for all $N\leqslant r\leqslant mk$,
\begin{center}
$\displaystyle\frac{1}{r}\sum\limits_{i=0}^{r-1}d(f(x_i),x_{i+1})<\epsilon$.
\end{center}
Putting $\epsilon_i=d(f(x_i),x_{i+1})$, $0\leqslant i\leqslant mk-1$, we have $\frac{1}{mk}\sum_{i=0}^{mk-1}\epsilon_i<\epsilon$. Define $y_j=x_{jk}$, $0\leqslant j\leqslant m$. We claim that $\{y_0,y_1,\dots,y_m\}$ is a $\delta$-average-chain of $f^k$ from $x$ to $y$. Note that using Lipschitz condition of $f$ with the fact that $L\geqslant 1$, we have\\
$d(f^k(y_0),y_1)=d(f^k(x_0),x_k)\leqslant d(f^k(x_0),f^{k-1}(x_1))+ d(f^{k-1}(x_1),f^{k-2}(x_2))+\dots+ d(f(x_{k-1}),x_k)\leqslant L^{k-1}\epsilon_0+L^{k-2}\epsilon_1+\dots+\epsilon_{k-1} \leqslant L^{k-1}(\epsilon_0+\epsilon_1+\dots+\epsilon_{k-1})$.\\
Similarly, one can prove that\\
$d(f^k(y_1),y_2)\leqslant L^{k-1}(\epsilon_k+\epsilon_{k+1}+\dots+\epsilon_{2k-1})$.\\
$\dots$\\
$d(f^k(y_{m-1}),y_m)\leqslant L^{k-1}(\epsilon_{(m-1)k}+\epsilon_{(m-1)k+1}+\dots+\epsilon_{mk-1})$.\\
This in turn gives
\begin{eqnarray*}
\displaystyle\frac{1}{m}\sum\limits_{i=0}^{m-1}d(f^k(y_i),y_{i+1})&\leqslant&\frac{1}{m}L^{k-1}(\epsilon_0+\epsilon_1+\dots+\epsilon_{mk-1})\\ &<&\frac{1}{m}L^{k-1}(\epsilon mk)=L^{k-1}(\epsilon k)=\delta.
\end{eqnarray*}
Thus $f^k$ is average chain transitive.
\end{proof}

\begin{thm}
Let $(X,f)$ be a dynamical system. If $f$ is average chain mixing, then $f\times f$ is average chain transitive.
\end{thm}
\begin{proof}
Consider the metric $d^*((x_1,y_1),(x_2,y_2))=\max\{d(x_1,x_2), d(y_1,y_2)\}$, $x_1,x_2,\\y_1,y_2\in X$ on $X\times X$. Let $\delta>0$ and $(a,b)$, $(c,d)\in X\times X$. Then there exist $n_1$, $n_2\in\mathbb{N}$ such that for any $p\geqslant n_1$ and any $q\geqslant n_2$, there are $\delta/2$-average-chains of $f$ from $a$ to $c$ and $b$ to $d$ of length $p$ and $q$ respectively. Taking $n_0=\max\{n_1,n_2\}$ we have $\{x_0=a,x_1,\dots,x_{n_0}=c\}$ and $\{y_0=b,y_1,\dots,y_{n_0}=d\}$ are $\delta/2$-average-chains of $f$, that is, there exist $N_1$, $N_2\in\mathbb{N}$, $N_1$, $N_2\leqslant n_0$, such that

\begin{center}
$\displaystyle\frac{1}{m}\sum\limits_{i=0}^{m-1}d(f(x_i),x_{i+1})<\delta/2\hspace{4mm}$ for all $N_1\leqslant m\leqslant n_0$
\end{center}
and
\begin{center}
$\displaystyle\frac{1}{r}\sum\limits_{i=0}^{r-1}d(f(y_i),y_{i+1})<\delta/2\hspace{4mm}$ for all $N_2\leqslant r\leqslant n_0$.
\end{center}
We claim that $\{(x_i,y_i)\}_{0\leqslant i\leqslant n_0}$ is the required $\delta$-average-chain of $f\times f$ from $(a,b)$ to $(c,d)$. Taking $N=\max\{N_1,N_2\}$ we have for all $N\leqslant n\leqslant n_0$,
\begin{eqnarray*}
\displaystyle\frac{1}{n}\sum\limits_{i=0}^{n-1}d^*((f\times f)(x_i,y_i),(x_{i+1},y_{i+1}))\!\!&\leqslant&\!\! \frac{1}{n}\sum\limits_{i=0}^{n-1}[d(f(x_i),x_{i+1})+d(f(y_i),y_{i+1})]\\ &<&\!\!(\delta/2+\delta/2)=\delta.
\end{eqnarray*}
Thus $f\times f$ is average chain transitive.
\end{proof}

\begin{thm}
Let $(X,f)$ be a dynamical system. If $f$ is totally average chain transitive, then $f\times f$ is average chain transitive.
\end{thm}

\begin{proof}
Let $\delta>0$ and $(a,b)$, $(c,d)\in X\times X$. Since $f$ is average chain transitive, there exists a $\delta/4$-average-chain of $f$ from $a$ to $c$, say $\{u_0=a,u_1,\dots,u_n=c\}$, that is, there exists $N_1\in\mathbb{N}$, $N_1\leqslant n$, such that
\begin{center}
$\displaystyle\frac{1}{r}\sum\limits_{i=0}^{r-1}d(f(u_i),u_{i+1})<\delta/4\hspace{4mm}$ for all $N_1\leqslant r\leqslant n$.
\end{center}
This gives
\begin{center}
$\displaystyle\frac{1}{n}\sum\limits_{i=0}^{n-1}d(f(u_i),u_{i+1})<\delta/4$.
\end{center}
Again using the average chain transitivity of $f$, there exists a $\delta/4$-average-chain of $f$ from $c$ to $c$, say $\{c=v_0,v_1,\dots,v_p=c\}$, that is, there exists $N_2\in\mathbb{N}$, $N_2\leqslant p$, such that

\begin{center}
$\displaystyle\frac{1}{r}\sum\limits_{i=0}^{r-1}d(f(v_i),v_{i+1})<\delta/4\hspace{4mm}$ for all $N_2\leqslant r\leqslant p$. \end{center}\newpage
This gives
\begin{center}
$\displaystyle\frac{1}{p}\sum\limits_{i=0}^{p-1}d(f(v_i),v_{i+1})<\delta/4$.
\end{center}
Since $f$ is totally average chain transitive, $f^p$ is average chain transitive so that there exists a $\delta/2$-average-chain of $f^p$ from $f^n(b)$ to $d$, say $\{f^n(b)=y_0,y_1,\dots,y_q=d\}$. Therefore

\begin{center}
$\displaystyle\frac{1}{q}\sum\limits_{i=0}^{q-1}d(f^p(y_i),y_{i+1})<\delta/2$.
\end{center}
Define
\begin{center}
$x_{pm+j}=\left\{\begin{array}{ll} f^j(y_m),\hspace{1mm} 0\leqslant m\leqslant q-1, 0\leqslant j\leqslant p-1,\\ y_q,\hspace{8.5mm} m=q, j=0,\end{array}\right.$
\end{center}
that is, $\{x_i\}_{0\leqslant i\leqslant pq}=\{f^n(b)=y_0, f(y_0),\dots, f^{p-1}(y_0), y_1, f(y_1),\dots, f^{p-1}(y_1),\dots,\\ y_{q-1},f(y_{q-1}),\dots, f^{p-1}(y_{q-1}), y_q\}$. Now define

\begin{center}
$z_i=\left\{\begin{array}{ll} f^i(b),\hspace{1mm} 0\leqslant i\leqslant n-1,\\ x_{i-n},\hspace{1mm} n\leqslant i\leqslant n+pq,\end{array}\right.$
\end{center}
that is, $\{z_i\}_{0\leqslant i\leqslant n+pq}=\{b,f(b),\dots,f^{n-1}(b),f^n(b)=y_0, f(y_0),\dots, f^{p-1}(y_0),y_1,\\ f(y_1),\dots, f^{p-1}(y_1),\dots, y_{q-1},f(y_{q-1}),\dots, f^{p-1}(y_{q-1}), y_q=d\}$. Consider
\begin{eqnarray*}
\frac{1}{n+pq}\sum\limits_{i=0}^{n+pq-1}d(f(z_i),z_{i+1})&=& \frac{1}{n+pq}\sum\limits_{j=0}^{q-1}d(f^p(y_j),y_{j+1})\\ &<& \frac{1}{q}\sum\limits_{j=0}^{q-1}d(f^p(y_j),y_{j+1})<\delta/2.
\end{eqnarray*}
Thus $\{z_i\}_{0\leqslant i\leqslant n+pq}$ is a $\delta/2$-average-chain of $f$ from $b$ to $d$.

\vspace{2mm}
Now define $\{w_i\}_{0\leqslant i\leqslant n+pq}=\{a=u_0,u_1,\dots,u_n=c,\underbrace{v_1,v_2,\dots,v_p=c}_{q\mbox{-times}},v_1,v_2,\\\dots,v_p=c,\dots,v_1,v_2,\dots,v_p=c\}$. Consider

\begin{eqnarray*}
\frac{1}{n+pq}\!\!\!\!\sum\limits_{i=0}^{n+pq-1}\!\!\!\!d(f(w_i),w_{i+1})\!\!\!&=&\!\!\! \frac{1}{n+pq}\Big[\sum\limits_{j=0}^{n-1}d(f(u_j),u_{j+1})+q\sum\limits_{k=0}^{p-1}d(f(v_k),v_{k+1})\Big]\\ &<&\!\!\! \frac{1}{n}\sum\limits_{j=0}^{n-1}d(f(u_j),u_{j+1})+\frac{q}{pq}\sum\limits_{k=0}^{p-1}d(f(v_k),v_{k+1})\\ &<&\!\!\!(\delta/4+\delta/4)=\delta/2.
\end{eqnarray*}
Thus $\{w_i\}_{0\leqslant i\leqslant n+pq}$ is a $\delta/2$-average-chain of $f$ from $a$ to $c$.\\ Hence $\{(w_i,z_i)\}_{0\leqslant i\leqslant n+pq}$ is a $\delta$-average-chain of $f\times f$ from $(a,b)$ to $(c,d)$ under the metric $d^*$.
\end{proof}

\section{Almost average shadowing property}

Let $(X,f)$ be a dynamical system.
\begin{Def}
For $\delta>0$, a sequence $\{x_i\}_{i\geqslant 0}$ in $X$ is said to be \textit{almost $\delta$-average-pseudo-orbit} of $f$ if
\begin{center}
$\limsup\limits_{n\rightarrow\infty}\displaystyle\frac{1}{n}\sum\limits_{i=0}^{n-1}d(f(x_i),x_{i+1})<\delta$.
\end{center}
\end{Def}

\begin{Def}
For $\epsilon>0$, an almost $\delta$-average-pseudo-orbit $\{x_i\}_{i\geqslant 0}$ of $f$ is said to be \textit{$\epsilon$-shadowed} in average by a point $x\in X$ if
\begin{center}
$\limsup\limits_{n\rightarrow\infty}\displaystyle\frac{1}{n}\sum\limits_{i=0}^{n-1}d(f^i(x),x_{i})<\epsilon$.
\end{center}
\end{Def}

\begin{Def}
A map $f$ is said to have the \textit{almost average shadowing property} (ALASP) if for any $\epsilon>0$, there is a $\delta>0$ such that every almost $\delta$-average-pseudo-orbit $\{x_i\}_{i\geqslant 0}$ of $f$ is $\epsilon$-shadowed in average by some point in $X$.
\end{Def}
From the definition, it follows that the ALASP implies the ASP.
\begin{rmk}
Consider the space $X_1$ and the map $f_1$ as given in \cite[Example 9.1]{K1}. By similar arguments as given in \cite[Theorem 9.2]{K1} one can prove that $f_1$ has the ASP but does not have the ALASP.
\end{rmk}
Clearly, constant maps have the ALASP. Following is an example of a map which does not have the ALASP.

\begin{ex}\label{ex2}
Consider $X=\{a,b\}$ with discrete metric $d$ and $f$ as the cycle permutation of $X$ defined by $f(a)=b$, $f(b)=a$. Fix $\epsilon=1/3$. Consider the following finite sequences\\ $y_0=\{a,b\}$,\\ $y_1=\{a,b,b,a\}$,\\ $y_2=\{a,b,a,b,b,a,b,a\}$,\\ $\dots$\\
$y_n=\{\underbrace{a,b,\dots,a,b}_{2^{n}},
\underbrace{b,a,\dots,b,a}_{2^{n}}\}$,
\\$\dots$\\ Take $\{x_i\}_{i\geqslant 0}=\{y_0\vee y_1\vee\dots\vee y_n\vee\dots\}$, where $\vee$ denotes the chain join. For example, $y_0\vee y_1=\{a,b,a,b,b,a\}$. Let $K_n=\sum\limits_{j=1}^{n}2^j=2(2^n-1)$, $n\in\mathbb{N}$. Then
\begin{center}
$\displaystyle\frac{1}{K_n}\sum\limits_{i=0}^{K_n-1}d(f(x_i),x_{i+1})=\frac{2(n-1)}{2(2^n-1)}\rightarrow 0$ as $n\rightarrow\infty$
\end{center}
which implies
\begin{center}
$\limsup\limits_{n\rightarrow\infty}\displaystyle\frac{1}{K_n}\sum\limits_{i=0}^{K_n-1}d(f(x_i),x_{i+1})=0$.
\end{center}
This gives $\{x_i\}_{i\geqslant 0}$ is an almost $\delta$-average-pseudo-orbit of $f$ for every $\delta>0$.
Now for $z\in\{a,b\}$, suppose $z=a$ (the case $z=b$ follows similarly).\\

$x_i=a,b$, $\;a,b,b,a$, $\;a,b,a,b,b,a,b,a$, $\;a,b,a,b,a,b,a,b,b,a,b,a,b,a,b,a\dots$\\
$f^i(z)=a,b$, $\;a,b,a,b$, $\;a,b,a,b,a,b,a,b$, $\;a,b,a,b,a,b,a,b,a,b,a,b,a,b,a,b\dots$\\
Then
\begin{center}
$\displaystyle\frac{1}{K_n}\sum\limits_{i=0}^{K_n-1}d(f^i(z),x_i)=\frac{2(2^{n-1}-1)}{2(2^n-1)}\rightarrow\frac{1}{2}$ as $n\rightarrow\infty$
\end{center}
which gives
\begin{center}
$\limsup\limits_{n\rightarrow\infty}\displaystyle\frac{1}{K_n}\sum\limits_{i=0}^{K_n-1}d(f^i(z),x_i)=\frac{1}{2}>\epsilon$.
\end{center}
Thus $f$ does not have the ALASP.
\end{ex}
\begin{thm}\label{T3}
Let $(X,f)$ be a dynamical system. If $f$ has the ALASP, then so does $f^k$ for every $k>1$.
\end{thm}
\begin{proof}

Let $k>1$ and $\epsilon>0$. Suppose $\delta>0$ is obtained for $\epsilon/k$ by the ALASP of $f$. Let $\{y_i\}_{i\geqslant 0}$ be an almost $\delta$-average-pseudo-orbit of $f^k$. Then
\begin{center}
$\limsup\limits_{n\rightarrow\infty}\displaystyle\frac{1}{n}\sum\limits_{i=0}^{n-1}d(f^k(y_i),y_{i+1})<\delta$.
\end{center}
Define $x_{mk+j}=f^j(y_m)$ for $m\geqslant 0$, $0\leqslant j\leqslant k-1$. Then for $n\in\mathbb{N}$, there exist $m\geqslant 0$ and $0\leqslant j\leqslant k-1$ such that $n=mk+j$. Therefore we have
\begin{eqnarray*}
\limsup\limits_{n\rightarrow\infty}\frac{1}{n}\sum\limits_{i=0}^{n-1}d(f(x_i),x_{i+1})&=& \limsup\limits_{m\rightarrow\infty}\frac{1}{mk+j}\sum\limits_{i=0}^{mk+j-1}d(f(x_i),x_{i+1})\\ &=& \limsup\limits_{m\rightarrow\infty}\frac{1}{mk+j}\sum\limits_{i=0}^{m}d(f^k(y_i),y_{i+1})\\ &\leqslant& \limsup\limits_{m\rightarrow\infty}\frac{1}{m+1}\sum\limits_{i=0}^{m}d(f^k(y_i),y_{i+1})<\delta.
\end{eqnarray*}
This gives $\{x_i\}_{i\geqslant 0}$ is an almost $\delta$-average-pseudo-orbit of $f$ so that there exists $z\in X$ such that

\begin{center}
$\limsup\limits_{n\rightarrow\infty}\displaystyle\frac{1}{n}\sum\limits_{i=0}^{n-1}d(f^i(z),x_{i})<\epsilon/k$.
\end{center}
Now consider
\begin{eqnarray*}
\limsup\limits_{n\rightarrow\infty}\frac{1}{n}\sum\limits_{i=0}^{n-1}d((f^k)^i(z),y_i)&=& \limsup\limits_{n\rightarrow\infty}\frac{1}{n}\sum\limits_{i=0}^{n-1}d(f^{ki}(z),x_{ki})\\ &\leqslant& \limsup\limits_{n\rightarrow\infty}\frac{1}{n}\sum\limits_{i=0}^{n-1}\sum\limits_{j=0}^{k-1} d(f^{ki+j}(z),x_{ki+j})\\ &=& k\limsup\limits_{n\rightarrow\infty}\frac{1}{nk}\sum\limits_{r=0}^{nk-1}d(f^r(z),x_{r})<\epsilon.
\end{eqnarray*}
Thus $f^k$ has the ALASP.
\end{proof}

\begin{thm}
Let $(X,f)$, $(Y,g)$ be two bounded dynamical systems with metric $d_1$, $d_2$ respectively. If $f$ and $g$ have the ALASP, then so does $f\times g$.
\end{thm}

\begin{proof}
Consider the metric $d^*((x_1,y_1),(x_2,y_2))=\max\{d_1(x_1,x_2), d_2(y_1,y_2)\}$, $x_1$, $x_2\in X$, $y_1$, $y_2\in Y$ on $X\times Y$. Let $\epsilon>0$ and $D=\mbox{diam}(X\times Y)$ which is finite, $X$, $Y$ being bounded. Choose $0<\eta<\epsilon$. Suppose $\delta_1>0$, $\delta_2>0$ are obtained for $\eta^2/{(2D+1)}^2$ by the ALASP of $f$, $g$ respectively. Let $\delta=\min\{\delta_1,\delta_2\}$ and $\{(x_i,y_i)\}_{i\geqslant 0}$ be an almost $\delta$-average-pseudo-orbit of $f\times g$. Then

\begin{center}
$\limsup\limits_{n\rightarrow\infty}\displaystyle\frac{1}{n}\sum\limits_{i=0}^{n-1}d^*((f\times g)(x_i,y_i),(x_{i+1},y_{i+1}))<\delta$.
\end{center}
This gives
\begin{center}
$\limsup\limits_{n\rightarrow\infty}\displaystyle\frac{1}{n}\sum\limits_{i=0}^{n-1}d_1(f(x_i),x_{i+1})<\delta\hspace{4mm}$ and
$\hspace{4mm}\limsup\limits_{n\rightarrow\infty}\displaystyle\frac{1}{n}\sum\limits_{i=0}^{n-1}d_2(g(y_i),y_{i+1})<\delta$.
\end{center}
Therefore there exist $x\in X$ and $y\in Y$ such that
\begin{center}
$\limsup\limits_{n\rightarrow\infty}\displaystyle\frac{1}{n}\sum\limits_{i=0}^{n-1}d_1(f^i(x),x_i)<\eta^2/(2D+1)^2$
\end{center}
and
\begin{center}
$\limsup\limits_{n\rightarrow\infty}\displaystyle\frac{1}{n}\sum\limits_{i=0}^{n-1}d_2(g^i(y),y_i)<\eta^2/(2D+1)^2$.
\end{center}
Define the sets $A=\{i\in\mathbb{Z_+}:d_1(f^i(x),x_i)\geqslant\eta/(2D+1)\}$, $B=\{i\in\mathbb{Z_+}:d_2(g^i(y),y_i)\geqslant\eta/(2D+1)\}$ and $C=\{i\in\mathbb{Z_+}:d^*((f\times g)^i(x,y),(x_i,y_i))\geqslant\eta/(2D+1)\}$. Then
\begin{center}
$\displaystyle\frac{\eta^2}{(2D+1)^2}>\limsup\limits_{n\rightarrow\infty}\frac{1}{n}\sum\limits_{i=0}^{n-1}d_1(f^i(x),x_i)\geqslant u_d(A)\frac{\eta}{2D+1}$.
\end{center}
This gives $u_d(A)\leqslant\eta/(2D+1)$. Similarly, $u_d(B)\leqslant\eta/(2D+1)$. Since $u_d(C)\leqslant u_d(A)+u_d(B)$, $u_d(C)\leqslant 2\eta/(2D+1)$. For $n\in\mathbb{N}$,

\begin{center}
$\displaystyle\frac{1}{n}\sum\limits_{i=0}^{n-1}d^*((f\times g)^i(x,y),(x_{i},y_{i}))\leqslant\frac{1}{n}[(n-|\{0,1,\dots,n-1\}\cap C|)\eta/(2D+1)]$
\end{center}
\begin{center}
$+\displaystyle\frac{1}{n}[|\{0,1,\dots,n-1\}\cap C|]D=\frac{\eta}{2D+1}+\big(D-\frac{\eta}{2D+1}\big)\displaystyle\frac{1}{n}(|\{0,1,\dots,n-1\}\cap C|)$.
\end{center}
This in turn gives
\begin{eqnarray*}
\limsup\limits_{n\rightarrow\infty}\frac{1}{n}\sum\limits_{i=0}^{n-1}d^*((f\times g)^i(x,y),(x_{i},y_{i}))&\leqslant&\frac{\eta}{2D+1}+\big(D-\frac{\eta}{2D+1}\big)u_d(C) \\&\leqslant&\frac{(2D+1)}{(2D+1)}\eta=\eta<\epsilon.
\end{eqnarray*}
Thus $f\times g$ has the ALASP.
\end{proof}

We have the following result using \cite[Theorem 3.2]{Y}.
\begin{thm}\label{T1}
Let $(X,f)$ be a compact dynamical system. If $f$ has the ALASP and the minimal points of $f$ are dense in $X$, then $f$ is totally strongly ergodic.
\end{thm}

Following example justifies that the shadowing property need not imply the ALASP.

\begin{ex}
Consider the Cantor set $\Sigma_2=\{0,1\}^{\mathbb{N}}$ with the metric $d(x,y)=\inf\{2^{-k}:x_i=y_i\mbox{ for all }i<k\}$ for every $x=(x_1,x_2,\dots)$, $y=(y_1,y_2,\dots)\in\Sigma_2$. Then the identity map $I_{\Sigma_2}$ on $\Sigma_2$ has the shadowing property \cite{AH}. However, every point of $\Sigma_2$ is minimal point of $I_{\Sigma_2}$ which implies that $I_{\Sigma_2}$ does not have the ALASP, by Theorem \ref{T1}.
\end{ex}

The following result can be obtained using \cite[Lemma 3.1]{K1}. However, we have the following direct proof also.

\begin{thm}\label{T2}
Let $(X,f)$ be a compact dynamical system and $f$ be surjective. If $f$ has the ALASP, then $f$ is chain transitive. In particular, $CR(f)=X$.
\end{thm}
\begin{proof}
Let $\delta>0$ and $x$, $y\in X$. Since $f$ is uniformly continuous, for this $\delta$, there exists $\eta$, $0<\eta<\delta$, such that for all $a$, $b\in X$ with $d(a,b)<\eta$ we have $d(f(a),f(b))<\delta$. By the ALASP of $f$, for this $\eta$, there exists $\zeta>0$ such that every almost $\zeta$-average-pseudo-orbit of $f$ is $\eta/2$-shadowed in average by some point in $X$. Choose $N_0\in\mathbb{N}$ such that $\frac{D}{N_0}<\zeta$, where $D=\mbox{diam}(X)$. Since $f$ is surjective, let $y_{-j}\in f^{-1}(y_{-j+1})$, $1\leqslant j\leqslant N_0-1$, where $y_0=y$. For $k\in\mathbb{Z}_+$, define

\begin{center}
$w_{2kN_0+j}=f^j(x)$, $0\leqslant j\leqslant N_0-1$,
\end{center}
\begin{center}
$w_{(2k+1)N_0+j}=y_{-(N_0-1)+j}$, $0\leqslant j\leqslant N_0-1$.
\end{center}
So $\{w_i\}_{i\geqslant 0}\equiv\{x,f(x),\dots,f^{N_0-1}(x),y_{-(N_0-1)},y_{-(N_0-2)},\dots,y_{-1},y,x,f(x),\dots,\\f^{N_0-1}(x),y_{-(N_0-1)},y_{-{(N_0-2)}},\dots,y,\dots\}$.

\vspace{2mm}
Now for $mN_0\leqslant n<(m+1)N_0$, $m\in\mathbb{N}$,
\begin{center}
$\displaystyle\frac{1}{n}\sum\limits_{i=0}^{n-1}d(f(w_i),w_{i+1})\leqslant \frac{mD}{mN_0}=\frac{D}{N_0}$.
\end{center}
This gives
\begin{center}
$\limsup\limits_{n\rightarrow\infty}\displaystyle\frac{1}{n}\sum\limits_{i=0}^{n-1}d(f(w_i),w_{i+1})\leqslant\frac{D}{N_0}<\zeta$.
\end{center}
Therefore $\{w_i\}_{i\geqslant 0}$ is an almost $\zeta$-average-pseudo-orbit of $f$ so that there exists $z\in X$ such that
\begin{equation}\label{3}
\limsup\limits_{n\rightarrow\infty}\displaystyle\frac{1}{n}\sum\limits_{i=0}^{n-1}d(f^i(z),w_i)<\eta/2.
\end{equation}
Note that there are infinitely many $k\in\mathbb{Z}_+$ for which there exist $i_k\in\{2kN_0,\\ 2kN_0+1,\dots,(2k+1)N_0-1\}$, that is, $w_{i_k}\in\{x,f(x),\dots,f^{N_0-1}(x)\}$ such that\\ $d(f^{i_k}(z),w_{i_k})<\eta$ otherwise it will lead to
\begin{center}
$\limsup\limits_{n\rightarrow\infty}\displaystyle\frac{1}{n}\sum\limits_{i=0}^{n-1}d(f^i(z),w_i)\geqslant\eta/2$,
\end{center}
which contradicts (\ref{3}). Similarly, an analogous statement holds with $2k$ replaced by $2k+1$ and $f^j(x)$ replaced by $y_{-(N_0-1)+j}$ for $0\leqslant j\leqslant N_0-1$. Therefore we can choose two positive integers $k_0$, $m_0$ such that $i_{k_0}<j_{m_0}$, $w_{i_{k_0}}=f^{k_1}(x)$ for some $0<k_1<N_0$ satisfying $d(f^{i_{k_0}}(z),w_{i_{k_0}})<\eta$ and $w_{j_{m_0}}=y_{-m_1}$ for some $0<m_1<N_0$ satisfying $d(f^{j_{m_0}}(z),w_{j_{m_0}})<\eta$.\\ Thus the required $\delta$-chain of $f$ from $x$ to $y$ is as follows:\\
$\{x,f(x),\dots,f^{k_1-1}(x),f^{k_1}(x)=w_{i_{k_0}},f^{i_{k_0}+1}(z),f^{i_{k_0}+2}(z),\dots,f^{j_{m_0}-1}(z),w_{j_{m_0}}=y_{-m_1},y_{-m_1+1},\dots,y_{-1},y\}$. Hence $f$ is chain transitive.
\end{proof}

We recall that if $(X,f)$ is a compact dynamical system, then $f$ is totally chain transitive iff $f$ is chain mixing \cite{R}.

\begin{thm}\label{T4}
Let $(X,f)$ be a compact dynamical system and $f$ be surjective. If $f$ has the ALASP, then $f$ is chain mixing.
\end{thm}
\begin{proof}
The proof follows from Theorem \ref{T3}, Theorem \ref{T2} and the fact that total chain transitivity implies chain mixing.
\end{proof}

\begin{rmk}
Example \ref{ex2} justifies that average chain mixing need not imply the ALASP. Also, note that identity map on a compact connected metric space is chain mixing but does not have the ALASP, by Theorem \ref{T1}.
\end{rmk}

Following result shows that a map having the ALASP on a compact metric space has only one chain component, namely, its chain recurrent set.

\begin{thm}
Let $(X,f)$ be a compact dynamical system. If $f$ has the ALASP, then $CR(f)$ consists of a single chain component.
\end{thm}
\begin{proof}
Suppose that $\Omega_1$ and $\Omega_2$ are two distinct chain components of $f$ in $CR(f)$. Since $X$ is compact, $\Omega_1$, $\Omega_2$ are closed and $f(\Omega_1)\subseteq\Omega_1$, $f(\Omega_2)\subseteq\Omega_2$ which implies $\cap_{n\geqslant 0}f^n(\Omega_1)\ne\emptyset$ and $\cap_{n\geqslant 0}f^n(\Omega_2)\ne\emptyset$. Let $x\in\cap_{n\geqslant 0}f^n(\Omega_1)$ and $y\in\cap_{n\geqslant 0}f^n(\Omega_2)$. Then $x\in\Omega_1$ such that $x=f^i(w_i)$ for every $i>0$, for some $w_i\in\Omega_1$ and $y\in\Omega_2$ such that $y=f^j(z_j)$ for every $j>0$, for some $z_j\in\Omega_2$. Let $\delta>0$ be given. Choose $\zeta$ and $N_0$ as in Theorem \ref{T2}. Consider almost $\zeta$-average-pseudo-orbits $\{x,f(x),\dots,f^{N_0-1}(x),z_{N_0-1},f(z_{N_0-1}),\dots,f^{N_0-2}(z_{N_0-1}),y,x,f(x),\dots,f^{N_0-1}(x),\\\dots\}$ and $\{y,f(y),\dots,f^{N_0-1}(y),w_{N_0-1},f(w_{N_0-1}),\dots,f^{N_0-2}(w_{N_0-1}),x,y,f(y),\dots,\\f^{N_0-1}(y),\dots\}$. By similar arguments as given in Theorem \ref{T2} one can prove that $x$ and $y$ are chain equivalent. This gives $\Omega_1=\Omega_2$, which is a contradiction. Thus $CR(f)$ consists of a single chain component.
\end{proof}

\end{document}